\documentclass{amsart}

\pdfoutput=1  

\usepackage[margin=4cm,bottom=3.7cm]{geometry}

\usepackage[maxbibnames=99,  
           giveninits=true, 
           sortcites=true,  
           natbib=true,     
           backend=biber]{biblatex}
\usepackage{amssymb, amsmath}
\usepackage{amsthm}
\usepackage{graphicx}
\usepackage{enumerate}
\usepackage{mathtools}
\usepackage{enumitem}
\usepackage{bm}

\usepackage[utf8]{inputenc}
\usepackage{mathtools}
\usepackage{colonequals}
\usepackage{lmodern}

\usepackage{microtype}
\usepackage{todonotes}

\usepackage{tikz}
\usepackage{hyperref}

\newcommand{\Anti}{\operatorname{Anti}}
\newcommand{\curl}{\operatorname{curl}}
\newcommand{\Curl}{\operatorname{Curl}}
\newcommand{\conormal}{\bm{\partial}}
\newcommand{\dev}{\operatorname{dev}}
\renewcommand{\div}{\operatorname{div}}
\newcommand{\Div}{\operatorname{Div}}
\newcommand{\grad}{\operatorname{grad}}
\newcommand{\sym}{\operatorname{sym}}

\newcommand{\abs}[1]{\lvert #1 \rvert}

\newcommand{\R}{\mathbb R}

\newcommand{\ba}{\mathbf{a}}
\newcommand{\be}{\mathbf{e}}
\newcommand{\bn}{\mathbf{n}}
\newcommand{\bt}{\mathbf{t}}

\newtheorem{definition}{Definition}
\newtheorem{theorem}[definition]{Theorem}
\newtheorem{corollary}[definition]{Corollary}
\newtheorem{lemma}[definition]{Lemma}

\title[Conforming finite elements for $H(\sym \Curl)$ and $H(\dev \sym \Curl)$]
      {Conforming finite elements\\ for $H(\sym \Curl)$ and $H(\dev \sym \Curl)$}

\thanks{%
The author would like to thank Adam Sky and Patrizio Neff for the interesting discussions.
}

\author[Sander]{Oliver Sander}
\address{Oliver Sander\\
Technische Universität Dresden\\
Fakultät für Mathematik\\
Zellescher Weg 12--14\\
01069 Dresden\\
Germany}
\email{oliver.sander@tu-dresden.de}

\begin{document}

\begin{abstract}
 We construct conforming finite elements for the spaces $H(\sym \Curl)$ and
 $H(\dev \sym \Curl)$.  Those are spaces of matrix-valued functions with
 symmetric or deviatoric-symmetric $\Curl$ in a Lebesgue space, and they appear
 in various models of nonstandard solid mechanics. The finite elements
 are not $H(\Curl)$-conforming.  We show the construction, prove conformity
 and unisolvence, and point out optimal approximation error bounds.
\end{abstract}

\maketitle


\section{Introduction}

In \cite{lewintan_mueller_neff:2020}, \citeauthor{lewintan_mueller_neff:2020} introduced the spaces
\begin{align*}
 W^{1,p}(\sym \Curl; \Omega; \R^3)
 & \colonequals
 \big\{ P \in L^p(\Omega; \R^{3 \times 3}) \; : \; \sym \Curl P \in L^p(\Omega; \R^{3 \times 3}) \big\} \\
 \intertext{and}
 W^{1,p}(\dev \sym \Curl; \Omega; \R^3)
 & \colonequals
 \big\{ P \in L^p(\Omega; \R^{3 \times 3}) \; : \; \dev \sym \Curl P \in L^p(\Omega; \R^{3 \times 3}) \big\},
\end{align*}
where $L^p(\Omega;\R^{3 \times 3})$, $p \ge 1$, is the Lebesgue space of $\R^{3 \times 3}$-valued
functions that are $p$-integrable on a domain $\Omega$.  The operator $\Curl$ denotes the
classical $\curl$ operator acting row-wise on a matrix; to distinguish the two we write
the matrix form with an upper-case letter.
The operators $\dev$ and $\sym$ produce the deviatoric and symmetric parts
of a $3 \times 3$-matrix
\begin{equation*}
 \dev A \colonequals A - \frac{1}{3} \operatorname{trace} A
 \qquad \text{and} \qquad
 \sym A \colonequals \frac{1}{2} (A + A^T),
\end{equation*}
respectively.
For brevity, we will call the spaces above $H(\sym \Curl)$ and $H(\dev \sym \Curl)$
in this manuscript.

\citeauthor{lewintan_mueller_neff:2020} presented several potential applications from the field of solid mechanics.
For numerical simulations it is therefore of interest to construct conforming finite elements
for these spaces. As
\begin{equation}
\label{eq:hcurl_subsets}
 H(\Curl)
 \subset
 H(\sym \Curl)
 \qquad \text{and} \qquad
 H(\Curl)
 \subset
 H(\dev \sym \Curl),
\end{equation}
a possible candidate are finite elements that are row-wise $H(\curl)$-conforming
(for example, the Nédélec elements~\cite{nedelec:1980,nedelec:1986}).
However, the subset relations~\eqref{eq:hcurl_subsets} are strict, and more structure of the new spaces
can be captured by finite element spaces
that are larger, i.e., not necessarily subspaces of $H(\Curl)$.

The classical way to construct conforming finite elements for a particular Sobolev space is to combine piecewise
polynomials on a grid $\mathcal{T}$ by certain continuity conditions.  Define
\begin{equation*}
 \Anti : \R^3 \to \R^{3 \times 3}
 \qquad
 \Anti \ba
 \colonequals
 \begin{pmatrix}
  0 & -a_3 & a_2 \\
  a_3 & 0 & - a_1 \\
  -a_2 & a_1 & 0
 \end{pmatrix},
\end{equation*}
which implies
\begin{equation*}
 (\Anti \mathbf{a}) \mathbf{b} = \mathbf{a} \times \mathbf{b}
\end{equation*}
for every $\mathbf{a}, \mathbf{b} \in \R^3$.
On any element $T \in \mathcal{T}$, we can integrate by parts:
\begin{align*}
 \int_T Q : \sym \Curl U\,dx - \int_T \Curl \sym Q : U\,dx
 & =
 - \int_{\partial T} \sym (U \Anti \bn) Q\,d\sigma \\
\intertext{and}
 \int_T Q : \dev \sym \Curl U\,dx - \int_T \Curl \dev \sym Q : U\,dx
 & =
 - \int_{\partial T} \dev \sym (U \Anti \bn) Q\,d\sigma,
\end{align*}
if the functions $U, Q : T \to \R^{3 \times 3}$ are sufficiently smooth
(see \cite[Chapter~3]{lewintan_mueller_neff:2020}).
From these formulas, we get the following characterization result.

\begin{theorem}
\label{thm:jump_conditions}
 Let $\Omega$ be bounded.  A piecewise continuously differentiable function
 $U : \Omega \to \R^{3 \times 3}$ on a grid $\mathcal{T}$
 is in $H(\sym \Curl)$ if and only if
 \begin{equation}
 \label{eq:sym_jump_condition}
  \sym ([U] \Anti \bn) = 0
 \end{equation}
 on every inner face of the grid, where $[U]$ is the jump of $U$
 at the face, and $\bn$ is a face normal.
 The function is in $H(\dev \sym \Curl)$ if and only if
 \begin{equation*}
  \dev \sym ([U] \Anti \bn) = 0
 \end{equation*}
 on every inner face of the grid.
\end{theorem}

The set of matrices that fulfill~\eqref{eq:sym_jump_condition}, but not the
corresponding condition $[U] \Anti \bn = 0$ for $H(\Curl)$ is spanned by the
identity matrix.  Multiples of the identity therefore play a special role,
and are treated separately in the finite element construction.

We have
\begin{equation*}
 H(\sym \Curl) \subset H(\dev \sym \Curl),
\end{equation*}
and by Theorem~4.1 of~\cite{lewintan_mueller_neff:2020} the two spaces
are not equal when $\Omega$ is bounded.
Curiously, however, Observation~2.3 in~\cite{lewintan_mueller_neff:2020} shows that
 \begin{equation}
  \label{eq:symanti_equals_devsymanti}
  \sym (U \Anti \bn) = 0
  \qquad \iff \qquad
  \dev \sym (U \Anti \bn) = 0
 \end{equation}
for any $U \in \R^{3 \times 3}$ and $\bn \in \R^3$, and therefore Theorem~\ref{thm:jump_conditions}
leads to identical finite element spaces for $H(\sym \Curl)$ and $H(\dev \sym \Curl)$.
In the following we will therefore only consider elements that are $H(\sym \Curl)$-conforming,
which are then automatically $H(\dev \sym \Curl)$-conforming.
Construction of finite element spaces that are $H(\dev \sym \Curl)$-conforming but not $H(\sym \Curl)$-conforming
will require nonstandard ideas.

The modern treatment of $H(\curl)$-conforming and related finite element spaces is a particularly
beautiful part of numerical mathematics, because it fits into the framework of finite element
exterior calculus~\cite{arnold_falk_winther:2006}.  This framework builds on the observation that
the space $H(\curl)$ forms part of the de Rham complex for the
classical vector calculus operators $\grad$, $\curl$, and $\div$.
A similar construction that involves the spaces $H(\sym \Curl)$ or $H(\dev \sym \Curl)$
is currently under investigation.
Closely related is the $\div\Div$ complex of~\citet{pauly_zulehner:2020}, which deals with
\begin{equation*}
 H(\sym \Curl) \cap \{ U \; : \; \operatorname{trace} U = 0 \}
\end{equation*}
instead of $H(\sym \Curl)$ itself.  There is a recent construction of a
finite element subcomplex of the $\div \Div$ complex in~\cite{hu_liang_ma:2021},
which reproduces the exact-sequence property of the original complex.

Besides dealing only with trace-free matrix functions, the finite element functions
of~\cite{hu_liang_ma:2021} are remarkably complicated.
In this paper we construct $H(\sym \Curl)$-conforming finite elements of any approximation order
that are much simpler.
We state the construction for tetrahedra, but a generalization to hexahedra
is straightforward.
The elements use full polynomial spaces, and the degrees of freedom are certain directional
point evaluations.  We show well-posedness, conformity and $H(\Curl)$-nonconformity,
and optimal interpolation error bounds.
However, the elements do not fulfill any exact-sequence properties,
and no inf--sup condition is shown either.  This is because the applications envisioned
in~\cite{lewintan_mueller_neff:2020} are not saddle-point problems, and therefore
these structural properties are of lesser importance.  The gain is a vastly
simplified construction compared to~\cite{hu_liang_ma:2021}, and finite elements
with fewer degrees of freedom per element.

The continuity conditions~\eqref{thm:jump_conditions} force our finite element functions
to be (almost) continuous at the grid vertices.  This is not surprising, as some
related finite elements also require vertex continuity~\cite{arnold_awanou_winther:2008,hu_liang_ma:2021}.
If the grid is such that all faces meeting at a vertex have one of three normals,
then we can get $H(\sym \Curl)$-conformity with less vertex continuity.
This alternative construction is described in Chapter~\ref{sec:hexahedral_elements}.

\section{Conformity-preserving degrees of freedom}

The degrees of freedom of our element are certain point evaluations. As a preparatory step
we therefore first consider conditions that ensure conformity at a single point.

\subsection{Face degrees of freedom}

We begin by conditions for points on a face of an element.  To every face $F$ in the
(three-dimensional) grid we associate a unit normal vector $\bn$, whose orientation
does not change throughout this manuscript.  Although normality is not used by the
following lemma, the vector $\bn$ that appears there will later be that face normal.

\begin{lemma}[Conformity]
 \label{lem:face_degree_conformity}
 Let $\ba_1, \ba_2, \bn$ be a basis of $\R^3$, and let $U \in \R^{3 \times 3}$.
 If
 \begin{alignat*}{6}
   \ba_1^T U (\bn \times \ba_1) & = 0 && \qquad\qquad &
   \ba_2^T U (\bn \times \ba_2) & = 0 \\
   \ba_1^T U (\bn \times \ba_2) + \ba_2^T U (\bn \times \ba_1) & = 0\\
   \bn^T U (\bn \times \ba_1) & = 0 && &
   \bn^T U (\bn \times \ba_2) & = 0,
 \end{alignat*}
 then
 \begin{equation}
 \label{eq:face_degree_conformity}
  \sym (U \Anti \bn) = 0.
 \end{equation}
\end{lemma}

\begin{proof}
 Equation~\eqref{eq:face_degree_conformity} is equivalent to
 \begin{equation*}
  \sym (U \Anti \bn) : Q = 0
  \qquad
  \forall Q \in \R^{3 \times 3},
 \end{equation*}
 and this in turn is equivalent to
 \begin{equation}
 \label{eq:conformity_v3}
  (U \Anti \bn) : Q = 0
  \qquad
  \forall Q \in \mathbb{S}^3,
 \end{equation}
 where $\mathbb{S}^3$ is the set of symmetric $3 \times 3$ matrices.
 By linearity, it is sufficient to test~\eqref{eq:conformity_v3} only for six matrices
 $Q_1, \dots, Q_6 \in \R^{3 \times 3}$
 that form a basis of $\mathbb{S}^3$.  One such basis is
 \begin{alignat*}{9}
  Q_1 & = \ba_1 \otimes \ba_1 && \qquad &
  Q_2 & = \ba_2 \otimes \ba_2 && \qquad &
  Q_3 & = \sym(\ba_1 \otimes \ba_2) \\
  Q_4 & = \sym (\ba_1 \otimes \bn) && &
  Q_5 & = \sym (\ba_2 \otimes \bn) && &
  Q_6 & = \bn \otimes \bn.
 \end{alignat*}
 Indeed, these matrices are linearly independent.  To see this, let $\alpha_1,\dots,\alpha_6 \in \R$
 be such that
 \begin{equation}
  \label{eq:linear_independence}
  \sum_{i=1}^6 \alpha_i Q_i = 0.
 \end{equation}
 Then, write $Q_i = A Q_i^e A^T$, where $A$ is the matrix with columns $\ba_1$, $\ba_2$, $\bn$,
 and $Q_1^e, \ldots, Q_6^e$ are the canonical basis vectors
 \begin{alignat*}{9}
  Q_1^e & = \be_1 \otimes \be_1 && \qquad &
  Q_2^e & = \be_2 \otimes \be_2 && \qquad &
  Q_3^e & = \sym(\be_1 \otimes \be_2) \\
  Q_4^e & = \sym (\be_1 \otimes \be_3) && &
  Q_5^e & = \sym (\be_2 \otimes \be_3) && &
  Q_6^e & = \be_3 \otimes \be_3.
 \end{alignat*}
 Equation~\eqref{eq:linear_independence} implies $A(\sum_{i=1}^6 \alpha_i Q_i^e)A^T = 0$,
 and as $A$ is invertible and the $Q_i^e$ are trivially linearly independent,
 it follows that $\alpha_1 = \ldots = \alpha_6 = 0$.

 As it turns out, not all six test matrices are required.
 Note that for any two vectors $\mathbf{a}, \mathbf{b} \in \R^3$ we get
 \begin{equation}
 \label{eq:contraction_as_matrix_product}
  U \Anti \bn : (\mathbf{a} \otimes \mathbf{b})
  =
  \sum_{i,j=1}^3 (U \Anti \bn)_{ij} \cdot a_ib_j
  =
  \mathbf{a}^T (U \Anti \bn) \mathbf{b},
 \end{equation}
 and therefore
 \begin{equation}
 \label{eq:normal_normal_dof_is_zero}
  (U \Anti \bn) : Q_6
  =
  \bn^T (U \Anti \bn) \bn
  =
  0.
 \end{equation}
 As a consequence, the other five basis vectors are enough to ensure~\eqref{eq:face_degree_conformity}.
 Using~\eqref{eq:contraction_as_matrix_product} to compute
\begin{alignat*}{2}
  (U \Anti \bn) : Q_1  & = \ba_1^T (U \Anti \bn) \ba_1  && = \ba_1^T U (\bn \times \ba_1) \\
  (U \Anti \bn) : Q_2  & = \ba_2^T (U \Anti \bn) \ba_2  && = \ba_2^T U (\bn \times \ba_2) \\
  (U \Anti \bn) : Q_3  & = \frac{1}{2} \big[ \ba_1^T (U \Anti \bn) \ba_2 + \ba_2^T (U \Anti \bn) \ba_1 \big] \\
                                               & \hspace{0.28\textwidth} = \frac{1}{2} \big[ \ba_1^T U (\bn \times \ba_2) && + \ba_2^T U (\bn \times \ba_1) \big] \\
  (U \Anti \bn) : Q_4  & = \frac{1}{2} \big[ \ba_1^T (U \Anti \bn) \bn + \bn^T (U \Anti \bn) \ba_1 \big]
                                              & & = \frac{1}{2} \bn^T U (\bn \times \ba_1) \\
  (U \Anti \bn) : Q_5  & = \frac{1}{2} \big[ \ba_2^T (U \Anti \bn) \bn + \bn^T (U \Anti \bn) \ba_2 \big]
                                              & & = \frac{1}{2} \bn^T U (\bn \times \ba_2)
\end{alignat*}
we get the assertion.
\end{proof}

Note how we would get the same set of equations when trying to satisfy
the seemingly weaker condition $\dev \sym (U \Anti \bn) = 0$.  In that case,
provided that $\ba_1$, $\ba_2$, $\bn$ have equal length,
one possible set of test matrices for the five-dimensional space
$\{ Q \in \mathbb{S}^3\; : \; \operatorname{trace} Q = 0 \}$ would be $Q_3$, $Q_4$, $Q_5$ from above,
together with $\ba_1 \otimes \ba_1 - \ba_2 \otimes \ba_2$ and
$\ba_2 \otimes \ba_2 - \bn \otimes \bn$.
However, by~\eqref{eq:normal_normal_dof_is_zero}, testing with
the latter is equivalent to testing with $\ba_2 \otimes \ba_2 = Q_2$, which, together with
$\ba_1 \otimes \ba_1 - \ba_2 \otimes \ba_2$ spans the same space as $Q_1$ and $Q_2$.
In view of~\eqref{eq:symanti_equals_devsymanti} this is no surprise.

\begin{lemma}[Unisolvence]
 \label{lem:face_unisolvence}
 If the five conditions of Lemma~\ref{lem:face_degree_conformity} hold,
 and if additionally
 \begin{alignat*}{6}
  \ba_1^T U (\bn \times \ba_2) - \ba_2^T U (\bn \times \ba_1) & = 0  & & \qquad &
  \ba_1^T U \bn & = 0\\
  \ba_2^T U \bn & = 0 & & &
  \bn^T U \bn & = 0,
 \end{alignat*}
 then
 \begin{equation*}
  U = 0.
 \end{equation*}
\end{lemma}

\begin{proof}
By adding and subtracting the third condition of Lemma~\ref{lem:face_degree_conformity}
and the first condition of Lemma~\ref{lem:face_unisolvence} we find that
the two are equivalent to
\begin{equation*}
 \ba_1^T U (\bn \times \ba_2) = 0
 \qquad \text{and} \qquad
 \ba_2^T U (\bn \times \ba_1) = 0.
\end{equation*}
Let $A$ be the matrix with columns $\ba_1$, $\ba_2$, $\bn$, and $B$ the matrix
with columns $\ba_1 \times \bn$, $\ba_2 \times \bn$, $\bn$.
Then the nine conditions can be written in matrix form
\begin{equation}
 \label{eq:face_conformity_matrix_form}
 A^T U B = 0.
\end{equation}
$A$ is invertible because the columns $\ba_1$, $\ba_2$, $\bn$ form a basis.
To see that $B$ is invertible, note first that both $\bn \times \ba_1$ and $\bn \times \ba_2$
are in the plane orthogonal to $\bn$.  Secondly, they are not collinear, because
otherwise $\ba_1$, $\ba_2$, $\bn$ would all be in a common plane.
Multiplying~\eqref{eq:face_conformity_matrix_form} with $A^{-T}$ from the left
and with $B^{-1}$ from the right then yields the assertion.
\end{proof}

\subsection{Edge degrees of freedom}
\label{sec:edge_dofs}

We now consider values on an edge $E$ of an element.
Here, the conformity conditions for the two adjacent faces $F_1$ and $F_2$ interact.
We equip every edge of the grid with a unit tangent vector $\bt_E$, and two
further vectors $\bn_{E,1}$ and $\bn_{E,2}$ that span the plane orthogonal to $E$.
For each pair $(F_i,E)$ of face $F_i$, $i=1,2$ and adjacent edge $E$ we define a unit conormal
$\conormal_i \colonequals \bt_E \times \bn_i$.  The conormal $\conormal_i$
is tangent to $F_i$ and orthogonal to $\bt_E$.

To state conformity-ensuring conditions we again use Lemma~\ref{lem:face_degree_conformity}.
For face $F_1$ with normal $\bn_1$ we set $\ba_1 = \bt_E$ and $\ba_2 = \conormal_1$.
For face $F_2$ with
normal $\bn_2$ we set $\ba_1 = \bt_E$ and $\ba_2 = \conormal_2$.  The conditions
for the two faces are:
\begin{center}
 \begin{tikzpicture}
  \node at (-1,3.3) {$F_1$};
  \node at (4.5,3.3) {$F_2$};

  \draw (-3.5,3.0) -- ++(10.5,0);

  \node at (-1.3,2.6)  {$\bt_E^T U (\bn_1 \times \bt_E) = 0$};
  \node at (4.8,2.6)   {$\bt_E^T U (\bn_2 \times \bt_E) = 0$};

  \node (dof2left) at (-1.3,1.95) {$\conormal_1^T U (\bn_1 \times \conormal_1) = 0$};
  \node (dof2right) at (4.8,1.95) {$\conormal_2^T U (\bn_2 \times \conormal_2) = 0$};

  \node at (-1.7,1.3)  {$\bt_E^T U (\bn_1 \times \conormal_1) + \conormal_1^T U (\bn_1 \times \bt_E) = 0$};
  \node at (5.2,1.3)   {$\bt_E^T U (\bn_2 \times \conormal_2) + \conormal_2^T U (\bn_2 \times \bt_E) = 0$};

  \node at (-1.3,0.65) {$\bn_1^T U (\bn_1 \times \bt_E) = 0$};
  \node at (4.8, 0.65) {$\bn_2^T U (\bn_2 \times \bt_E) = 0$};

  \node (dof5left) at (-1.3,0) {$\bn_1^T U (\bn_1 \times \conormal_1) = 0$};
  \node (dof5right) at (4.8,0) {$\bn_2^T U (\bn_2 \times \conormal_2) = 0$.};

  \draw [<->] (dof2left) edge [out=0,in=180] (dof5right);
  \draw [<->] (dof5left) edge [out=0,in=180] (dof2right);
 \end{tikzpicture}

\end{center}
Together, these are more conditions than there are variables, but
as it turns out we can unify the two pairs of conditions marked by arrows.
We first assign the others to the faces meeting at $E$.  We get
\begin{itemize}
 \item For the face $F_1$:
  \begin{subequations}
  \begin{align}
   \label{eq:edge_face_dof_F11}
   \bt_E^T U (\bn_1 \times \bt_E) & = 0\\
   \label{eq:face_dof_F}
   \bt_E^T U (\bn_1 \times \conormal_1) + \conormal_1^T U (\bn_1 \times \bt_E) & = 0\\
   \label{eq:edge_face_dof_F12}
   \bn_1^T U (\bn_1 \times \bt_E) & = 0
  \end{align}
  \end{subequations}

 \item For the face $F_2$:
  \begin{subequations}
  \begin{align}
   \label{eq:edge_face_dof_F21}
   \bt_E^T U (\bn_2 \times \bt_E) & = 0 \\
   \label{eq:face_dof_tildeF}
   \bt_E^T U (\bn_2 \times \conormal_2) + \conormal_2^T U (\bn_2 \times \bt_E) & = 0\\
   \label{eq:edge_face_dof_F22}
   \bn_2^T U (\bn_2 \times \bt_E) & = 0
  \end{align}
  \end{subequations}
\end{itemize}
As $\bn_1 \times \conormal_1 = \bn_2 \times \conormal_2 = \bt_E$ by definition of the conormal,
the remaining four conditions are equivalent to requiring that $U \bt_E \in \R^3$
is collinear to $\bt_E$. This can be stated without explicit reference
to the faces $F_1$ and $F_2$ by replacing the four conditions with
\begin{align}
\label{eq:edge_edge_conditions}
 \bn_{E,1}^T U \bt_E = 0
 \qquad\qquad
 \bn_{E,2}^T U \bt_E = 0,
\end{align}
where $\bn_{E,1}$ and $\bn_{E,2}$ are the two vectors associated to the edge $E$ that span
the normal space of $\bt_E$.

By construction, for either face, the face and edge conditions are enough to control conformity.

\begin{lemma}[Conformity]
\label{lem:edge_conformity}
 Let $U \in \R^{3 \times 3}$, and let $F$ be a face with normal $\bn$,
 bordering the edge $E$.
 If the three face conditions
 \begin{align*}
  \bt_E^T U (\bn \times \bt_E) & = 0\\
  \bt_E^T U (\bn \times \conormal) + \conormal^T U (\bn \times \bt_E) & = 0\\
  \bn^T U (\bn \times \bt_E) & = 0
 \end{align*}
 and the two edge conditions~\eqref{eq:edge_edge_conditions} hold,
 then $\sym (U \Anti \bn) = 0$.
\end{lemma}

\begin{proof}
 The edge conditions~\eqref{eq:edge_edge_conditions} are equivalent to
 \begin{equation*}
  \conormal^T U (\bn \times \conormal) = 0
  \qquad \qquad
  \bn^T U (\bn \times \conormal) = 0.
 \end{equation*}
 Together with the three face conditions they form a set of conditions
 as given by Lemma~\ref{lem:face_degree_conformity} for the three vectors $\bt_E$, $\conormal$, $\bn$.
 The assertion then follows from Lemma~\ref{lem:face_degree_conformity}.
\end{proof}

However, the eight conditions for $F_1$, $F_2$, and $E$ together are not enough to uniquely determine
the value of $U \in \R^{3 \times 3}$.  The joint kernel is spanned by
the identity matrix.  Indeed, let $U$ be the identity
matrix.  Then $\bt_E^T U (\bn_1 \times \bt_E) = \bt_E^T(\bn_1 \times \bt_E) = 0$,
and likewise for the five other conditions consisting of only one addend.
For the remaining condition for face $F_1$ we get
\begin{align*}
 \bt_E^T U(\bn_1 \times \conormal_1) + \conormal_1^T U (\bn_1 \times \bt_E)
 & =
 \bt_E^T (\bn_1 \times \conormal_1) + \bt_E^T (\conormal_1 \times \bn_1)
 = 0,
\end{align*}
by invariance under circular shift of the triple product,
and likewise for face $F_2$. We therefore need one further condition
to control multiples of the identity matrix.  A suitable choice is
\begin{equation}
\label{eq:identity_dof}
 \bt_E^T U \bt_E = 0.
\end{equation}
This will later turn into a degree of freedom assigned to the element.

\begin{lemma}[Unisolvence]
\label{lem:edge_unisolvence}
 Let $U \in \R^{3 \times 3}$ be such that the
 conditions~\eqref{eq:edge_face_dof_F11}--\eqref{eq:edge_face_dof_F12},
 \eqref{eq:edge_face_dof_F21}--\eqref{eq:edge_face_dof_F22},
 the two conditions~\eqref{eq:edge_edge_conditions}, and
 the condition~\eqref{eq:identity_dof} holds. Then $U = 0$.
\end{lemma}
\begin{proof}
Suppose that all nine conditions hold.
Subtracting~\eqref{eq:identity_dof} from~\eqref{eq:face_dof_F}
and~\eqref{eq:face_dof_tildeF} yields
\begin{equation*}
 \bt_E^T U \bt_E = 0
 \qquad \qquad
 \conormal_1^T U (\bn_1 \times \bt_E) = 0
 \qquad \qquad
 \conormal_2^T U (\bn_2 \times \bt_E) = 0.
\end{equation*}
These nine conditions can then be written as three vector equations
\begin{align*}
 \begin{pmatrix} \bt_E^T \\ \bn_1^T \\ \conormal_1^T \end{pmatrix}
 U
 (\bn_1 \times \bt_E) & = 0 \in \R^3 \\
 \begin{pmatrix} \bt_E^T \\ \bn_2^T \\ \conormal_2^T \end{pmatrix}
 U
 (\bn_2 \times \bt_E) & = 0 \in \R^3 \\
 \begin{pmatrix} \bt_E^T \\ \bn_{E,1}^T \\ \bn_{E,2}^T \end{pmatrix}
 U \bt_E & = 0 \in \R^3.
\end{align*}
Each of the matrices on the left is invertible, and we therefore conclude that
\begin{align*}
 U (\bn_1 \times \bt_E)
 =
 U (\bn_2 \times \bt_E)
 =
 U \bt_E = 0 \in \R^3.
\end{align*}
This can be rewritten as
\begin{equation*}
 U \Big(\bn_1 \times \bt_E \;\big|\; \bn_2 \times \bt_E \;\big|\; \bt_E\Big) = 0 \in \R^{3 \times 3},
\end{equation*}
and since the vectors $\bn_1 \times \bt_E$, $\bn_2 \times \bt_E$ and $\bt_E$
are linearly independent we obtain that $U = 0$.
\end{proof}

\subsection{Vertex degrees of freedom}
\label{sec:vertex_dofs}

Let $V$ be a vertex of an element, and let $F_1$, $F_2$, $F_3$ be the faces that meet at $V$.
The three corresponding face normals $\bn_1$, $\bn_2$, $\bn_3$ form a basis, and
we can therefore formulate the conformity conditions
of Lemma~\ref{lem:face_degree_conformity} for each of the three faces
in terms of $\bn_1$, $\bn_2$, and $\bn_3$.  As two conditions can be shared
for each pair of adjacent faces, we obtain nine degrees of freedom in total,
and a natural assignment to the edges and faces at the vertex~$V$:

\begin{center}
 \begin{tikzpicture}
  \fill [gray!20] (-2, 1) -- (0,0) -- (2,1) -- cycle;
  \fill [gray!30] (-2, 1) -- (0,0) -- (0,-1.6) -- cycle;
  \fill [gray!10] ( 2, 1) -- (0,0) -- (0,-1.6) -- cycle;

  \draw (-2, 1) -- (0,0);
  \draw ( 2, 1) -- (0,0);
  \draw ( 0, -1.6) -- (0,0);

  \node at (0,0.6) {$F_1$};
  \node at ( 0,  2.8) {$\bn_2^T U (\bn_1 \times \bn_3) + \bn_3^T U (\bn_1 \times \bn_2) = 0$};

  \node at (-0.7,-0.3) {$F_2$};
  \node at (-4.3, -0.8) {$\bn_3^T U (\bn_2 \times \bn_1) + \bn_1^T U (\bn_2 \times \bn_3) = 0$};

  \node at (0.7,-0.3) {$F_3$};
  \node at ( 4.3, -0.8) {$\bn_1^T U (\bn_3 \times \bn_2) + \bn_2^T U (\bn_3 \times \bn_1) = 0$};

  \node at (-2.8, 1.9) {$\bn_1^T U (\bn_1 \times \bn_2) = 0$};
  \node at (-2.8, 1.3) {$\bn_2^T U (\bn_1 \times \bn_2) = 0$};

  \node at ( 2.8, 1.9) {$\bn_1^T U (\bn_1 \times \bn_3) = 0$};
  \node at ( 2.8, 1.3) {$\bn_3^T U (\bn_1 \times \bn_3) = 0$};

  \node at (   0, -2) {$\bn_2^T U (\bn_2 \times \bn_3) = 0$};
  \node at (   0, -2.6) {$\bn_3^T U (\bn_2 \times \bn_3) = 0$};

 \end{tikzpicture}

\end{center}

Using the same trick as in Section~\ref{sec:edge_dofs}, the edge conditions can be
formulated in terms of the edge tangents $\bt_{E_i}$ and edge normal vectors
$\bn_{E_i,1}$, $\bn_{E_i,2}$.  However, the nine conditions do not form a unisolvent
set, because they do not control the identity matrix.
This is because the three face degrees of freedom are not independent---any two
imply the third one.

A natural way out would be to take any two, and assign them to the vertex itself.
However, for an unstructured grid it is unclear how to do this in a way that is independent of the geometry
of~$T$.  We are therefore forced to abandon the approach and force full continuity
at the vertex---with the exception of the subspace spanned by the identity matrix.
To this end we set up the continuity conditions
\begin{equation}
\label{eq:vertex_offdiagonal_conditions}
\begin{aligned}
 U_{12} & = 0  & \qquad U_{13} & = 0 & \qquad  U_{23} & = 0 \\
 U_{21} & = 0  &        U_{31} & = 0 &         U_{32} & = 0
\end{aligned}
\end{equation}
for the off-diagonals, together with
\begin{equation}
\label{eq:vertex_diagonal_conditions}
 U_{11} - U_{22} = 0
 \qquad \qquad
 U_{22} - U_{33} = 0
\end{equation}
for the diagonal entries, and
\begin{equation}
\label{eq:vertex_identity_condition}
 U_{11} + U_{22} + U_{33} = 0
\end{equation}
for multiples of the identity.

The first two sets of conditions then directly imply conformity.
\begin{lemma}[Conformity]
 \label{lem:vertex_conformity}
 Let $U \in \R^{3 \times 3}$ be such that Conditions~\eqref{eq:vertex_offdiagonal_conditions}
 and~\eqref{eq:vertex_diagonal_conditions} hold.
 Then $\sym (U \Anti \bn) = 0$ für any $\bn \in \R^3$.
\end{lemma}

Together with~\eqref{eq:vertex_identity_condition}, which controls multiples
of the identity matrix, we obtain a unisolvent set of degrees of freedom.

\begin{lemma}[Unisolvence]
\label{lem:vertex_unisolvence}
 Let $U \in \R^{3 \times 3}$ be such that
 Conditions~\eqref{eq:vertex_offdiagonal_conditions}, \eqref{eq:vertex_diagonal_conditions},
 and~\eqref{eq:vertex_identity_condition} hold.  Then $U = 0$.
\end{lemma}

\section{A tetrahedral element based on complete polynomials}
\label{sec:tetrahedral_element}

We can now construct the finite element for tetrahedra.  It is based on a full polynomial space,
and the degrees of freedom are directional point evaluations of the types developed
in the previous chapter.
In what follows, $\Pi_k(T; \R^{3 \times 3})$ is the space of $\R^{3 \times 3}$-valued polynomials
of order~$k$ or less on a tetrahedron~$T$. The definition makes use of a
set of Lagrange points on~$T$, which should have the usual layout.
Note again that the same construction also works for hexahedral elements and the ansatz space
$\Pi^\times_k$ of $\R^{3 \times 3}$-valued functions that are $k$th-order polynomials in each local
coordinate direction.

Remember that we define a unit normal $\bn$ for each face of the grid, and a basis of
unit vectors $\bt_E$, $\bn_{E,1}$, $\bn_{E,2}$ for each edge $E$ such that $\bt_E$ is tangent
to $E$, and $\bn_{E,1}$, $\bn_{E_2}$ span the orthogonal space of $E$.  The orientation of
each of these vectors is arbitrary but fixed for the entire grid.

\begin{definition}
\label{def:simplex_element}
 Let $T$ be a tetrahedron in $\R^3$.  The $k$th-order $H(\sym \Curl)$ finite element on $T$ is the space
 of all functions $U \in \Pi_k(T;\R^{3 \times 3})$, with the following degrees of freedom:
  \begin{enumerate}
  \item
   For each vertex $V$ of $T$:
   \begin{enumerate}
    \item \label{enum:vertex_offdiagonal_dof}
      The off-diagonal values of $U$ at $V$
      \begin{align*}
       U \mapsto U(V)_{ij}
       \qquad
       i,j=1,2,3, \quad i \neq j,
      \end{align*}
      ($4 \times 6$ degrees of freedom),

    \item \label{enum:vertex_diagonal_dof}
      the differences between the diagonal entries
      \begin{equation*}
       U \mapsto U(V)_{11} - U(V)_{22}
       \qquad \text{and} \qquad
       U \mapsto U(V)_{22} - U(V)_{33},
      \end{equation*}
      ($4 \times 2$ degrees of freedom),

    \item \label{enum:vertex_identity_dof}
     the trace of $U$ at $V$
     \begin{equation*}
      U \mapsto U(V)_{11} + U(V)_{22} + U(V)_{33}
     \end{equation*}
     (4 degrees of freedom).

   \end{enumerate}

  \item
   If $k \ge 2$: For each edge $E$ of $T$, and each inner Lagrange point $L$ on $E$:
   \begin{enumerate}
    \item \label{enum:edge_edge_dof}
     The quantities
     \begin{equation*}
      U \mapsto \bn_{E,1}^T U(L) \bt_E
      \qquad \text{and} \qquad
      U \mapsto \bn_{E,2}^T U(L) \bt_E,
     \end{equation*}
     ($6 \times (k-1) \times 2$ degrees of freedom),

    \item \label{enum:edge_face_dof}
     for each of the two adjacent faces $F_i$, $i=1,2$, the quantities
     \begin{align*}
      U & \mapsto \bt_E^T U(L) (\bn_i \times \bt_E)
      \qquad
      U \mapsto \bn_i^T U(L) (\bn_i \times \bt_E)
      \\
      U & \mapsto \bt_E^T U(L) (\bn_i \times \conormal_i) + \conormal_i^T U(L) (\bn_i \times \bt_E)
     \end{align*}
     where $\conormal_i$ is the conormal to $F_i$ at $L$, with orientation such
     that $\bt_E = \bn_i \times \conormal_i$.
     ($6 \times (k-1) \times 2 \times 3$ degrees of freedom),

    \item \label{enum:edge_identity_dof}
     the quantity
     \begin{equation*}
      U \mapsto \bt_E^T U(L) \bt_E
     \end{equation*}
      ($6 \times (k-1)$ degrees of freedom).

   \end{enumerate}

  \item
   If $k \ge 3$:
   For each face $F$ of $T$ with normal $\bn$ and two vectors $\ba_1$, $\ba_2$
   (not necessarily tangent, but such that $\ba_1$, $\ba_2$, $\bn$
   are linearly independent):
   \begin{enumerate}
    \item \label{enum:face_dof_normal}
   For each inner Lagrange point $L$ of $F$ the quantities
   \begin{alignat*}{6}
     U & \mapsto \ba_1^T U(L) (\bn \times \ba_1) & \qquad &
     U   \mapsto \ba_2^T U(L) (\bn \times \ba_2) \\
     U & \mapsto \ba_1^T U(L) (\bn \times \ba_2) & + \ba_2^T U(L) (\bn \times \ba_1)\\
     U & \mapsto \bn^T U(L) (\bn \times \ba_1) & &
     U   \mapsto \bn^T U(L) (\bn \times \ba_2) & &
   \end{alignat*}
   ($4 \times \binom{k-1}{2} \times 5$ degrees of freedom),

  \item \label{enum:face_dof_tangential}
   the quantities
   \begin{alignat*}{6}
    U & \mapsto \ba_1^T U (\bn \times \ba_2) - \ba_2^T U (\bn \times \ba_1) & \qquad &
    U \mapsto \ba_1^T U \bn \\
    U & \mapsto \ba_2^T U \bn & &
    U \mapsto \bn^T U \bn
   \end{alignat*}
   ($4 \times \binom{k-1}{2} \times 4$ degrees of freedom).
  \end{enumerate}

  \item \label{enum:inner_dof}
   If $k\ge 4$: For each inner Lagrange point $L$ of $T$: The entries of $U(L)$
   ($\binom{k-1}{3} \times 9$ degrees of freedom).
 \end{enumerate}
\end{definition}

Note that the list
contains $\binom{k+3}{3} \times 9$ degrees of freedom in total, which is precisely the
dimension of the polynomial space $\Pi_k(T; \R^{3 \times 3})$.  To prove unisolvence we therefore show
that $U \in \Pi_k(T; \R^{3 \times 3})$ is zero if all degrees of freedom are zero.

\begin{theorem}
\label{thm:global_unisolvence}
 Let $U \in \Pi_k(T;\R^{3 \times 3})$ be such that all degrees of freedom
 of Definition~\ref{def:simplex_element} are zero.
 Then $U \equiv 0$.
\end{theorem}

\begin{proof}
 We simply show that $U$ is zero at all Lagrange points. This is shown in
 Lemmas~\ref{lem:face_unisolvence}, \ref{lem:edge_unisolvence},
 and~\ref{lem:vertex_unisolvence} for face, edge, and vertex degrees
 of freedom, respectively.  For interior degrees of freedom it follows
 directly from the definition.
\end{proof}

At the same time, the degrees of freedom allow to control the required conformity.

\begin{theorem}[Conformity]
\label{thm:global_conformity}
 Let $F$ be a face of $T$ with normal $\bn$.  Let $U \in \Pi_k(T;\R^{3 \times 3})$
 be such that the degrees of freedom of types~\ref{enum:vertex_offdiagonal_dof}),
 \ref{enum:vertex_diagonal_dof}),
 \ref{enum:edge_edge_dof}), \ref{enum:edge_face_dof}),
 and~\ref{enum:face_dof_normal}) corresponding to $F$ are zero.  Then
 \begin{equation*}
  \sym (U \Anti \bn) = 0
 \end{equation*}
 on $F$.
\end{theorem}

\begin{proof}
 The function $\sym (U \Anti \bn)$ is a polynomial of degree $k$.  Therefore
 its restriction to~$F$ is
 determined uniquely by its values on the $\binom{k+2}{2}$
 Lagrange nodes on $F$.  The assertion holds for the face boundary Lagrange points
 by Lemmas~\ref{lem:edge_conformity} and~\ref{lem:vertex_conformity}, and for the inner points
 of $F$ by Lemma~\ref{lem:face_degree_conformity}.
\end{proof}

From their description in Definition~\ref{def:simplex_element}, the degrees
of freedom have natural associations to faces, edges, etc.\ of the tetrahedron $T$.
However, in the definition of the global finite element space we only identify
those degrees of freedom that control conformity:

\begin{definition}
\label{def:finite_element_space}
 Let $\mathcal{T}$ be a conforming grid of $\Omega$.
 The finite element space $V_{h,k}^\textup{sym\,Curl}$ is the set
 \begin{equation*}
  V_{h,k}^\textup{sym\,Curl}
  \colonequals
  \Big\{ U_h \in L^2(\Omega; \R^{3\times 3})
    \; : \;
    U_h \in \Pi_k(T;\R^{3\times 3}) \quad \forall T \in \mathcal{T}
  \Big\},
 \end{equation*}
 with the restriction that
 \begin{itemize}
  \item vertex degrees of freedom of types~\ref{enum:vertex_offdiagonal_dof})
       and~\ref{enum:vertex_diagonal_dof}) coincide for elements that share the vertex,
  \item edge degrees of freedom of types~\ref{enum:edge_edge_dof})
    and~\ref{enum:edge_face_dof}) coincide for elements that share the edge,
  \item face degrees of freedom of type~\ref{enum:face_dof_normal})
    coincide for elements that share the face.
 \end{itemize}
\end{definition}

Combining Theorem~\ref{thm:global_conformity} with the characterization result
of Theorem~\ref{thm:jump_conditions}, we directly get the following conformity relation.

\begin{corollary}
  $V_{h,k}^\textup{sym\,Curl} \subset H(\sym \Curl) \subset H(\dev \sym \Curl)$.
\end{corollary}

Observe that Definition~\ref{def:finite_element_space} does \emph{not} require type~\ref{enum:face_dof_tangential})
degrees of freedom or degrees of freedom related to identity matrices
(types~\ref{enum:vertex_identity_dof}) and~\ref{enum:edge_identity_dof})) to match
for adjacent elements, even though they are presented as belonging to the
element boundary in Definition~\ref{def:simplex_element}.
These are the degrees of freedom that allow to violate $H(\Curl)$-conformity.

\begin{theorem}
 For any $k \ge 1$ and any connected grid $\mathcal{T}$ with more than one element,
 there is a function $U_h \in V_{h,k}^\textup{sym\,Curl}$ such that $U_h \notin H(\Curl)$.
\end{theorem}
\begin{proof}
Let $U_h : \Omega \to \R^{3 \times 3}$ be a function that is zero everywhere
except on one element $T$ of $\mathcal{T}$, where $U_h$ is the identity matrix.
This is a finite element function in the sense of Definition~\ref{def:finite_element_space}.
Indeed, on $T$, the only nonzero degrees of freedom are
of types~\ref{enum:vertex_identity_dof}), \ref{enum:edge_identity_dof}), \ref{enum:face_dof_tangential}),
and~\ref{enum:inner_dof}),
and hence the coupling restrictions of Definition~\ref{def:finite_element_space}
are fulfilled.
The function $U_h$ is not in $H(\Curl)$, because on any face of $T$
with normal $\bn$ we have
\begin{equation*}
 [U_h] \Anti \bn = \Anti \bn \neq 0.
 \qedhere
\end{equation*}
\end{proof}

Interpolation error bounds for the space $V_{h,k}^\text{sym\,Curl}$ are standard,
because the polynomial space is invariant under affine transformations, and the degrees of freedom
are essentially point evaluations.  Optimal bounds for the interpolation
error therefore follow from the standard arguments~\cite{bartels:2016}.

Let $\mathcal{T}$ be a tetrahedral grid, and let $W_{\mathcal{T},m,p}$ be the space
of all functions $U \in H(\sym \Curl)$ such that for each $T \in \mathcal{T}$,
the restriction $U|_T$ is in $W^{m,p}(T;\R^{3 \times 3})$.
For these functions, if $m \ge 2$ and $p \ge 2$, the degrees of freedom
of Definition~\ref{def:simplex_element} imply a well-defined interpolation operator.
\begin{theorem}
 Let $I_h : W_{\mathcal{T},m,p} \to V_{h,k}^\textup{sym\,Curl}$ be the interpolation
 operator associated to the degrees of freedom of Definition~\ref{def:simplex_element}.
 For any $U \in W_{\mathcal{T},m,p}$ and $0 \le k \le m$, there is a constant $c>0$
 independent from the grid resolution and quality, such that
 \begin{equation*}
  \abs{U - I_h U}_{W^{k,p}(\Omega}
  \le
  c \max_{T \in \mathcal{T}} h_T^m \rho_T^{-k} \abs{U}_{W^{m,p}(\Omega}),
 \end{equation*}
 where $h_T$ is the diameter of $T$, and $\rho_T$ is its incircle radius.
\end{theorem}

\section{A larger space on hexahedral grids with few normals}
\label{sec:hexahedral_elements}

The tight conformity conditions at the vertices can be loosened a bit if all faces
meeting at a vertex are normal to one of three vectors $\bn_{V,1}$, $\bn_{V,2}$, $\bn_{V,3}$
associated to that vertex.
In that case, these three vectors can be used to define global vertex
degrees of freedom that do not couple (almost) the entire matrix
as in Chapter~\ref{sec:vertex_dofs}.

Grids with the required property necessarily consist of hexahedral elements only.
By far the most important example are axis-aligned grids, but there are more,
in particular if curvilinear hexahedra are used.

\subsection{Vertex degrees of freedom revisited}

Let again $V$ be a vertex of an element.  We now assume that the vertex $V$
comes equipped with a basis of unit vectors $\bn_{V,1}$, $\bn_{V,2}$, $\bn_{V,3}$, and that
for any element $H$ adjacent to $V$, the three faces $F_1$, $F_2$, $F_3$ of $H$ at $V$ are
orthogonal to $\bn_{V,1}$, $\bn_{V,2}$, $\bn_{V,3}$, respectively.
Under these new circumstances we retry the failed attempt of Chapter~\ref{sec:vertex_dofs},
and formulate the conformity conditions
of Lemma~\ref{lem:face_degree_conformity} for each of the three faces
in terms of $\bn_{V,1}$, $\bn_{V,2}$, and $\bn_{V,3}$.  Omitting the subindex $V$
in the following we obtain nine degrees of freedom in total,
and a natural assignment to the edges and faces at the vertex~$V$:

\begin{center}
 \begin{tikzpicture}
  \fill [gray!20] (-2, 1) -- (0,0) -- (2,1) -- cycle;
  \fill [gray!30] (-2, 1) -- (0,0) -- (0,-1.6) -- cycle;
  \fill [gray!10] ( 2, 1) -- (0,0) -- (0,-1.6) -- cycle;

  \draw (-2, 1) -- (0,0);
  \draw ( 2, 1) -- (0,0);
  \draw ( 0, -1.6) -- (0,0);

  \node at (0,0.6) {$F_1$};
  \node at ( 0,  2.8) {$\bn_2^T U (\bn_1 \times \bn_3) + \bn_3^T U (\bn_1 \times \bn_2) = 0$};

  \node at (-0.7,-0.3) {$F_2$};
  \node at (-4.3, -0.8) {$\bn_3^T U (\bn_2 \times \bn_1) + \bn_1^T U (\bn_2 \times \bn_3) = 0$};

  \node at (0.7,-0.3) {$F_3$};
  \node at ( 4.3, -0.8) {$\bn_1^T U (\bn_3 \times \bn_2) + \bn_2^T U (\bn_3 \times \bn_1) = 0$};

  \node at (-2.8, 1.9) {$\bn_{E_3,1}^T U \bt_{E_3} = 0$};
  \node at (-2.8, 1.3) {$\bn_{E_3,2}^T U \bt_{E_3} = 0$};

  \node at ( 2.8, 1.9) {$\bn_{E_2,1}^T U \bt_{E_2} = 0$};
  \node at ( 2.8, 1.3) {$\bn_{E_2,2}^T U \bt_{E_2} = 0$};

  \node at (   0, -2)   {$\bn_{E_1,1}^T U \bt_{E_1} = 0$};
  \node at (   0, -2.6) {$\bn_{E_1,2}^T U \bt_{E_1} = 0$};

 \end{tikzpicture}
\end{center}
In contrast to the illustration in Chapter~\ref{sec:vertex_dofs},
we have here formulated the edge conditions in terms of the edge
geometries.  For $i=1,2,3$, the vector $\bt_{E_i}$ denotes the unit tangent
vector of the edge $E_i$ opposite of $F_i$, and $\bn_{E_i,1}$, $\bn_{E_i,2}$
are the given vectors that span the orthogonal plane of $E_i$.
That way, the edge quantities can be used as global edge degrees of freedom.

The special geometric situation comes into play for the three
face conditions.  Remember that they are not independent.
However, the fact that the normals $\bn_1$, $\bn_2$, $\bn_3$ are now the same
for all elements adjacent to $V$ allows to turn the face conditions into
vertex degrees of freedom.
For this, we pick any two of the three face conditions, e.g.,
\begin{align}
\label{eq:joint_vertex_conditions_a}
 \bn_2^T U (\bn_1 \times \bn_3) + \bn_3^T U (\bn_1 \times \bn_2) & = 0 \\
 \intertext{and}
\label{eq:joint_vertex_conditions_b}
 \bn_3^T U (\bn_2 \times \bn_1) + \bn_1^T U (\bn_2 \times \bn_3) & = 0,
\end{align}
and we associate them to $V$.
This means that these quantities will be the same for all elements at $V$,
which is possible because by assumption all elements that meet at $V$ share the same
three normal vectors.

The following conformity result then
again follows from Lemma~\ref{lem:face_degree_conformity}:
\begin{lemma}[Conformity]
 Let $U \in \R^{3 \times 3}$, and let $F_i$, $i =1,2,3$, be a face at a vertex $V$,
 with normal $\bn_i$. If the four edge conditions hold for the two edges adjacent to $F_i$ and $V$,
 and if additionally the conditions~\eqref{eq:joint_vertex_conditions_a}
 and~\eqref{eq:joint_vertex_conditions_b} hold, then $\sym (U \Anti \bn_i) = 0$.
\end{lemma}
\begin{proof}
 For simplicity we show the assertion for $F_1$ only.  The four edge conditions
 \begin{alignat*}{3}
  \bn_{E_2,1}^T U \bt_{E_2} & = 0  & \qquad \qquad & \bn_{E_2,2}^T U \bt_{E_2} & = 0 \\
  \bn_{E_3,1}^T U \bt_{E_3} & = 0  &               & \bn_{E_3,2}^T U \bt_{E_3} & = 0
 \end{alignat*}
 signify that $U \bt_{E_2}$ is collinear to to $\bt_{E_2}$, and that $U\bt_{E_3}$
 is collinear to $\bt_{E_3}$.  As $\bn_1$ and $\bn_3$ are both normal to $\bt_{E_2}$
 (without being collinear to each other),  we get the equivalent conditions
 \begin{alignat}{3}
 \label{eq:vanilla_edge_conditions_hexa_a}
  \bn_1^T U (\bn_1 \times \bn_3) & = 0  & \qquad \qquad & \bn_3^T U (\bn_1 \times \bn_3) & = 0, \\
  \intertext{and likewise}
 \label{eq:vanilla_edge_conditions_hexa_b}
  \bn_1^T U (\bn_1 \times \bn_2) & = 0  &               & \bn_2^T U (\bn_1 \times \bn_2) & = 0.
 \end{alignat}
 Together with~\eqref{eq:joint_vertex_conditions_a}, $\sym (U \Anti \bn_1) = 0$ then
 follows from Lemma~\ref{lem:face_degree_conformity}.

 The proof for face $F_2$ is identical.  For $F_3$, the required fifth condition is
 the sum of~\eqref{eq:joint_vertex_conditions_a} and~\eqref{eq:joint_vertex_conditions_b}.
\end{proof}

To complement the eight conditions to a unisolvent set
we need to additionally control the identity matrix. One suitable condition for this is
\begin{equation}
\label{eq:vertex_identity_condition_hexa}
 \bn_1^T U (\bn_2 \times \bn_3) + \bn_2^T U (\bn_1 \times \bn_3) + \bn_3^T U (\bn_1 \times \bn_2) = 0.
\end{equation}
This condition will later be assigned to the grid element itself, and not
shared across vertices.

\begin{lemma}[Unisolvence]
 Let $U \in \R^{3 \times 3}$ be such that the six edge conditions of a vertex~$V$,
 as well as conditions~\eqref{eq:joint_vertex_conditions_a}, \eqref{eq:joint_vertex_conditions_b},
 and~\eqref{eq:vertex_identity_condition_hexa} hold.  Then $U = 0$.
\end{lemma}
\begin{proof}
The argument is similar to Lemma~\ref{lem:edge_unisolvence}:
Subtracting~\eqref{eq:joint_vertex_conditions_a} from~\eqref{eq:vertex_identity_condition_hexa}
leads to $\bn_1^T U (\bn_2 \times \bn_3) = 0$, and inserting this into~\eqref{eq:joint_vertex_conditions_b}
implies $\bn_3^T U (\bn_1 \times \bn_2) = 0$. Inserting this back into~\eqref{eq:joint_vertex_conditions_a}
yields $\bn_2^T U (\bn_3 \times \bn_1) = 0$.
Rewriting the edge conditions as in~\eqref{eq:vanilla_edge_conditions_hexa_a} or~\eqref{eq:vanilla_edge_conditions_hexa_b},
the resulting set of nine conditions can be
written as $A^T U B = 0$, where $A$ is the matrix with columns $\bn_1$, $\bn_2$, and $\bn_3$,
and $B$ is the matrix with columns $\bn_1 \times \bn_2$, $\bn_3 \times \bn_1$,
and $\bn_2 \times \bn_3$.  As both $A$ and $B$ are invertible, the assertion follows.
\end{proof}

\subsection{A hexahedral element with discontinuous vertices}

We can use the new set of vertex conditions to construct a larger finite element space
for hexahedral grids where all faces at a vertex $V$ have one of three normals
$\bn_{V,1}$, $\bn_{V,2}$, $\bn_{V,3}$.
Let $\Pi_k^\times(T; \R^{3 \times 3})$ be the space of $\R^{3 \times 3}$-valued polynomials
on a tetrahedron~$H$ that are of order $k$ or less in each of the three local
coordinate directions of $H$.

\begin{definition}
 Let $H$ be a hexahedron in $\R^3$.
 The $k$th-order $H(\sym \Curl)$ finite element on $H$ is the space
 of all functions $U \in \Pi_k^\times(H;\R^{3 \times 3})$, with the following degrees of freedom:
  \begin{enumerate}
  \item
   For each vertex $V$ of $H$:
   \begin{enumerate}
    \item \label{enum:vertex_conformity_dof}
      The quantities
      \begin{align*}
       U & \mapsto \bn_{V,2}^T U(V) (\bn_{V,1} \times \bn_{V,3}) + \bn_{V,3}^T U(V) (\bn_{V,1} \times \bn_{V,2}) \\
       U & \mapsto \bn_{V,3}^T U(V) (\bn_{V,2} \times \bn_{V,1}) + \bn_{V,1}^T U(V) (\bn_{V,2} \times \bn_{V,3}),
      \end{align*}
      where $\bn_{V,1}$, $\bn_{V,2}$, $\bn_{V,3}$ are the normal vectors of the three
      adjacent faces ($4 \times 2$ degrees of freedom),

    \item \label{enum:vertex_edge_dof}
      for each edge $E_i$, $i=1,2,3$, at $V$, the two quantities
      \begin{equation*}
       U \mapsto \bn_{E_i,1}^T U(V) \bt_{E_i}
       \qquad \text{and} \qquad
       U \mapsto \bn_{E_i,2}^T U(V) \bt_{E_i},
      \end{equation*}
      ($4 \times 3 \times 2 = 24$ degrees of freedom),

    \item
     the quantity
     \begin{equation*}
      U \mapsto \bn_{V,1}^TU(V)(\bn_{V,2} \times \bn_{V,3}) + \bn_{V,2}^T U(V) (\bn_{V,1} \times \bn_{V,3}) + \bn_{V,3}^T U(V)(\bn_{V,1} \times \bn_{V,2})
     \end{equation*}
     (4 degrees of freedom).

   \end{enumerate}

 \end{enumerate}
 Edge, face, and element degrees of freedom are as in Definition~\ref{def:simplex_element}.
\end{definition}

This makes $(k+1)^3 \times 9$ degrees of freedom in total, which equals the
dimension of the polynomial space $\Pi_k^\times(H; \R^{3 \times 3})$.
Unisolvence and conformity of this element are proved just as in
Theorems~\ref{thm:global_unisolvence} and~\ref{thm:global_conformity}.

In the definition of the global finite element space, local degrees of freedom are identified
across common vertices, edges or faces.

\begin{definition}
\label{def:finite_element_space_hexa}
 Let $\mathcal{T}$ be a grid such that at each vertex $V$, all adjacent faces are
 orthogonal to one of three unit vectors $\bn_{V,1}$, $\bn_{V,2}$, or $\bn_{V,3}$.
 The finite element space $H_{h,k}^\textup{sym\,Curl}$ is the set
 \begin{equation*}
  H_{h,k}^\textup{sym\,Curl}
  \colonequals
  \Big\{ U_h \in L^2(\Omega; \R^{3\times 3})
    \; : \;
    U_h \in \Pi_k^\times(H;\R^{3\times 3}) \quad \forall H \in \mathcal{T}
  \Big\},
 \end{equation*}
 with the restriction that
 \begin{itemize}
  \item vertex degrees of freedom of types~\ref{enum:vertex_conformity_dof})
       and~\ref{enum:vertex_edge_dof}) coincide for elements that share the vertex,
  \item edge degrees of freedom of types~\ref{enum:edge_edge_dof})
    and~\ref{enum:edge_face_dof}) coincide for elements that share the edge,
  \item face degrees of freedom of type~\ref{enum:face_dof_normal})
    coincide for elements that share the face.
 \end{itemize}
\end{definition}

This space differs from the one of Definition~\ref{def:finite_element_space} only
at the vertices.  To show that the new space is larger, simply count the degrees
of freedom at an inner grid vertex $V$.  By construction, eight elements meet at $V$.
Then Definition~\ref{def:finite_element_space} has 8 global vertex degrees of
freedom at $V$, and 8 further ones for the identity matrix components at $V$
of the 8 adjacent elements.
Definition~\ref{def:finite_element_space_hexa}, in contrast, has only two
global vertex degrees of freedom, the same 8 identity degrees of freedom,
and additionally two degrees of freedom for each of the six adjacent edges.
This makes 22 in total, in contrast to only 16 for Definition~\ref{def:finite_element_space}.

Using the same argument as in Chapter~\ref{sec:tetrahedral_element} we get conformity in $H(\sym \Curl)$
and $H(\Curl)$-nonconformity.

\begin{theorem}
 We have the subset relation
 \begin{equation*}
  H_{h,k}^\textup{sym\,Curl} \subset H(\sym \Curl) \subset H(\dev \sym \Curl).
 \end{equation*}
 On the other hand,
 for any $k \ge 1$ and any connected grid $\mathcal{T}$ with more than one element,
 there is a function $U_h \in H_{h,k}^\textup{sym\,Curl}$ such that $U_h \notin H(\Curl)$.
\end{theorem}

Optimal interpolation error bounds follow along the standard arguments as for the
tetrahedral element in Chapter~\ref{sec:tetrahedral_element}.

\printbibliography

\end{document}